\documentclass[a4paper, 11pt, reqno]{amsart}

\usepackage{amsmath, amsfonts, amssymb}
\usepackage{amsthm}
\usepackage{color}
\usepackage{amsaddr}
\usepackage{tikz}
\usetikzlibrary{arrows, backgrounds, decorations.pathreplacing}
\usepackage [autostyle, english = american]{csquotes}
\MakeOuterQuote{"}

\newcommand{\N}{\mathbb{N}}
\newcommand{\R}{\mathbb{R}}
\newcommand{\C}{\mathbb{C}}
\newcommand{\D}{\mathbb{D}}
\newcommand{\Barr}{\mathrm{B}}

\newtheorem{thm}{Theorem}
\newtheorem*{question}{Question}
\newtheorem{prop}[thm]{Proposition}
\newtheorem{lem}[thm]{Lemma}
\newtheorem{cor}[thm]{Corollary}
\newtheorem{dfn}[thm]{Definition}
\newtheorem{rmk}[thm]{Remark}

\title{Worm domains are not Gromov hyperbolic}
\author{Leandro Arosio$^1$, Gian Maria Dall'Ara$^2$, Matteo Fiacchi$^3$}
\address{$^1$ Department of Mathematics, University of Rome “Tor Vergata”\\
$^2$ Istituto Nazionale di Alta Matematica ``F. Severi", Research Unit SNS Pisa\\
$^3$ Faculty of Mathematics and Physics, University of Ljubljana}

\date{\today}

\subjclass[2010]{Primary 32F45; Secondary  32T20, 53C23}

\begin{document}

\maketitle

\begin{abstract}
We show that  Worm domains are not Gromov hyperbolic with respect to the Kobayashi distance.

\end{abstract}

\tableofcontents
\section{Introduction}

A central problem in contemporary several complex variables is to determine when a complete Kobayashi hyperbolic domain  $\Omega\subset \subset \C^n$ is Gromov hyperbolic when endowed with its Kobayashi distance.
Assume in what follows that $\Omega$ is smoothly bounded. 

Some families of relevant domains are Gromov hyperbolic: Balogh--Bonk \cite{BB} proved it for   strongly pseudoconvex domains, and Zimmer \cite{Zim1} showed it for  convex domains of D'Angelo finite type. The third-named author showed it  \cite{fia} for pseudoconvex domains of finite type in $\C^2$. On the other hand, Gaussier--Seshadri \cite{GS} proved that for smoothly bounded {\it convex} domains $\Omega\subset \subset \C^n$ an analytic disc in the boundary is an obstruction to Gromov hyperbolicity. This result was later strengthened by Zimmer \cite{Zim1}, who showed that the same is true if $\Omega$ is a smoothly bounded  $\C$-{\it convex} domain. The  following important question remains open.

\begin{question}
Is an analytic disc in the boundary an obstruction to Gromov hyperbolicity for    a  smoothly bounded complete Kobayashi hyperbolic domain $\Omega\subset \subset \C^n$?
 \end{question}
 
 In this paper we study the Gromov hyperbolicity of the Worm domains  introduced by Diederich--Forn\ae ss \cite{DF}, which have  a holomorphic annulus in the boundary and are highly non-$\C$-convex. 
 Worm domains play a central role in several complex variables as they provide counterexamples to several important questions. See, e.g., \cite{KP} for a review of  the properties of  Worm domains.
 We actually  consider a more general class of {\sl Worms} (see Definition \ref{defworm}), with an open  Riemann surface  in the boundary, and prove the following result:
\begin{thm}\label{main}
  Worms are not Gromov hyperbolic w.r.t. the Kobayashi distance.
\end{thm}
The   proof is based on Barrett's scaling (cf.~\cite[Section 4]{Barrett}). We rescale the Worm $W$ obtaining in the limit a holomorphic fiber bundle, which we call a {\sl pre-Worm}, with base an open hyperbolic Riemann surface and with fiber the right half-plane. We show that such a pre-Worm cannot be Gromov hyperbolic. Since the Kobayashi distance is continuous with respect to this scaling, this yields the result. 

\medskip {\bf Acknowledgements.} LA is partially supported by a MIUR Excellence Department Project awarded to the Department of Mathematics, University of Rome Tor Vergata, CUP E83C18000100006. GD acknowledges the support of Istituto Nazionale di Alta Matematica "F. Severi". MF is supported by the European Union (ERC Advanced grant HPDR, 101053085 to Franc Forstneri\v c). Finally, we thank Caterina Stoppato for communicating us a proof of Proposition \ref{kobcompl} for the Diederich--Forn\ae ss Worm.

\section{Gromov hyperbolicity -- Basic definitions}

In this section we will review some basic definitions and properties of Gromov hyperbolic spaces. The book \cite{BH} is one of the standard references.

\begin{dfn} 
Let $(X,d)$ be a metric space.
For every $x,y,o\in X$  the {\sl Gromov product} is 
$$(x|y)_o:=\frac{1}{2}[d(x,o)+d(y,o)-d(x,y)].$$

The metric space $(X,d)$ is $\delta$-{\sl hyperbolic} if for all $x,y,z,o\in X$
	$$(x|y)_o\geq\min\{(x|z)_o,(y|z)_o\}-\delta.$$
Finally, a metric space is {\sl Gromov hyperbolic} if it is $\delta$-hyperbolic for some $\delta\geq0$.
\end{dfn}

\begin{dfn}
Let $(X,d)$ be a metric space, $I\subset\R$ be an interval and $A\geq1$ and $B\geq0$. A function $\sigma:I\rightarrow X$ is
	\begin{enumerate}
		\item a {\sl geodesic} if for each $s,t\in I$ 
		$$d(\sigma(s),\sigma(t))=|t-s|; $$
		\item a $(A,B)$-{\sl quasigeodesic} if for each $s,t\in I$
		$$A^{-1}|t-s|-B\leq d(\sigma(s),\sigma(t))\leq A|t-s|+B.$$
	\end{enumerate}
	A $(A,B)$-{\sl quasigeodesic triangle} is a choice of three points in $X$ and three $(A,B)$-quasigeodesic segments connecting these points, called its {\sl sides}. If $M\geq 0$, a $(A,B)$-\textit{quasigeodesic triangle} is $M$-{\sl slim} if every side is contained in the $M$-neighborhood of the other two sides.
\end{dfn}

Finally, recall that a metric space $(X,d)$ is {\sl proper} if closed balls are compact, and 
 {\sl geodesic} if  any two points can be connected by a geodesic.
A fundamental property of  geodesic Gromov hyperbolic spaces is that  quasigeodesics are uniformly close to geodesics, a fact which implies the following characterization of Gromov hyperbolicity.

\begin{prop}\cite[Corollary 1.8]{BH}
A proper geodesic metric space $(X, d)$ is $\delta$-hyperbolic if and only if for all $A\geq1$ and $B\geq0$ there exists $M\geq 0$ such that every $(A,B)$-quasi-geodesic triangle is $M$-slim.
\end{prop}

\section{Worms and pre-Worms}

Let $X$ be an open Riemann surface, and let $\theta:X\rightarrow\R$ be a smooth ``angle'' function.
Consider the domain in $X\times \C$ defined as
\[
Z(X, \theta):=\{(z,w)\in X\times\C\colon \Re(we^{-i\theta(z)})>0 \}, 
\]
which is readily seen to be a smooth fiber bundle with base $X$ and fiber a half-plane. 

\begin{prop}\label{prp:hol_fiber_bundle}
If the function $\theta$ is harmonic, then  $Z(X, \theta)$ is a holomorphic fiber bundle. 
\end{prop}

\begin{proof}
Let $v$ be (minus) a local harmonic conjugate of $\theta$, so that $F(z)=v(z)+i\theta(z)$ is a holomorphic function on an open set $U\subset X$. Then $Z(X, \theta)$ is locally defined over $U$ by $\Re(we^{-F(z)})=\Re(we^{-v(z)-i\theta(z)})>0$, and $(z,w)\mapsto (z, e^{-F(z)}w)$ is the desired local trivialization. \end{proof}

\begin{dfn}[pre-Worms]
If the function $\theta$ is harmonic, we call the holomorphic fiber bundle  $Z(X, \theta)$  a {\sl pre-Worm}.
\end{dfn}
\begin{rmk}
Pre-Worms are sectorial domains in the sense of \cite{Barrett2} (see in particular Example 2.2).
\end{rmk}

A pre-Worm $Z(X,\theta)$ with hyperbolic base $X$ is  complete Kobayashi hyperbolic by the following classical result.
\begin{prop}[\cite{kobayascione}, Theorem 3.2.15]\label{prop:bundles_complete}
Let $\pi:E\rightarrow X$ be a holomorphic fiber bundle with fiber $F$. Assume that $F$ and $X$ are both (complete) Kobayashi hyperbolic. Then $E$ is (complete) Kobayashi hyperbolic. 
\end{prop}

Now we proceed to the definition of the Worms. First of all, given two compact intervals $I,J\subset \R$ such that $I\subset J^{\circ} $, we denote by $\eta:\R\rightarrow [0,+\infty)$ any smooth function satisfying the following properties: \begin{itemize}
	\item on $I$ the function  $\eta$ vanishes identically; 
	\item on  $\R\setminus I$ the function $\eta$ is real-analytic and  satisfies  $\eta''>0$
	(in particular $\eta$ is strictly positive and $\eta'\neq 0$ on $\R\setminus I$); 
	\item $J=\{\eta\leq 1\}$. 
	\end{itemize} 
The precise choice of a function $\eta$ satisfying the above properties is completely irrelevant for what follows. 

Next, given an open  Riemann surface $Y$ equipped with a smooth angle function $\theta:Y\rightarrow \R$ and two compact intervals $I,J$ as above, we define
\[
W:=\{(z,w)\in Y\times\C\colon |w-e^{i\theta(z)}|^2<1-\eta(\theta(z)) \}.
\]
We assume the following\label{theta_assumptions}: \begin{itemize}
	\item  $\theta$ \emph{has no critical points where $\theta(z)\in \partial I$ or $\theta(z)\in \partial J$}; 
	\item \emph{$\theta^{-1}(J)$ is a compact subset of $Y$}.
	\end{itemize}

\begin{prop}\label{prop:worm}
The domain $W\subset\subset  Y\times \C$ has smooth boundary. Moreover, if $\theta$ is harmonic, then $W$ is  Levi-pseudoconvex.
\end{prop}

\begin{proof}
The precompactness of the domain $W$ is a consequence of our assumption that $\theta^{-1}(J)$ is compact. The domain $W$ has defining function \[r(z,w)=w\overline{w}-we^{-i\theta(z)}-\overline{w}e^{i\theta(z)}+\eta(\theta(z)).
\] 
We show that $dr\neq 0$ for all $(z,w)\in \partial W$.   If  $\partial_{\bar{w}}r\neq 0$ this is clear.
 Since $\partial_{\bar{w}}r=w-e^{i\theta(z)}$ vanishes only if $w=e^{i\theta(z)}$, we may assume that this identity holds. Then necessarily $\eta(\theta(z))=1$, that is, $\theta(z)\in \partial J$, in which case \[
\partial_{\bar{z}} r=i\partial_{\bar{z}}\theta(z) we^{-i\theta(z)}-i\partial_{\bar{z}}\theta(z)\overline{w}e^{i\theta(z)}+\eta'(\theta(z))\partial_{\bar{z}}\theta(z) =\eta'(\theta(z))\partial_{\bar{z}}\theta(z)\neq 0
\] by our assumption about the critical points of $\theta$. This proves that $W$ has smooth boundary. 

Since Levi-pseudoconvexity is a local property, we may restrict the $z$ variable to an open set $U\subset Y$ where $\theta(z)$ admits a harmonic conjugate $v(z)$, as in the proof of Proposition \ref{prp:hol_fiber_bundle}. A local defining function for the boundary of $W$ is then given by \[
e^{-v}r = |e^{-\frac{F}{2}}w|^2-2\Re(we^{-F})+e^{-v}\eta\circ \theta, 
\]
where $F(z)=v(z)+i\theta(z)$ is holomorphic. Recalling that moduli squared (resp. real parts) of holomorphic functions are plurisubharmonic (resp. pluriharmonic), we see that $e^{-v}r$ equals a plurisubharmonic function plus $e^{-v}\eta\circ \theta$, which is a function of the variable $z$ alone. If we show that the latter is subharmonic, we are done. One computes \[
\Delta(e^{-v}\eta\circ \theta) = \Delta(e^{-v})\eta\circ \theta +2\nabla(e^{-v})\cdot \nabla(\eta\circ \theta)+ e^{-v}\Delta(\eta\circ \theta), 
\]
where $\Delta$ and $\nabla$ are the ordinary real Laplacian and gradient in $\C\equiv \R^2$.
%We claim  that for all $z\in U$ we have $\nabla(e^{-v})\cdot \nabla(\eta\circ \theta)=0$. 
In $U$ we have
$$\nabla(e^{-v})\cdot \nabla(\eta\circ \theta)=-e^{-v}(\eta'\circ \theta)\nabla v\cdot \nabla \theta=0,$$
by Cauchy--Riemann equations.

%In fact, this is clear if $\nabla \theta(z)=0$. If instead $\nabla \theta(z)\neq 0$, then $F=v+i\theta$  the level curves of $v$ and $\theta$ are transversal. 
Next, notice that $e^{-v}=|e^{-\frac{F}{2}}|^2$ is subharmonic. Since $\eta$ and $e^{-v}$ are nonnegative, all we are left to do to check the nonnegativity of $\Delta(e^{-v}\eta\circ \theta)$ is to verify that $\Delta(\eta\circ \theta)\geq 0$. By direct computation, we see that \[
\Delta(\eta\circ \theta) =4|\partial_{\bar{z}}\theta|^2 \eta''\circ \theta, 
\] 
which is nonnegative thanks to our convexity assumption on the auxiliary function $\eta$.
\end{proof}

\begin{dfn}[Worms]\label{defworm}
	If the function $\theta$ is harmonic (and satisfies the assumptions on page \pageref{theta_assumptions}), we call the domain $W\subset Y\times \C$  a {\sl Worm}.
\end{dfn}

The reader may find a picture of a Worm in Figure \ref{fig:worm}. 

	\begin{figure}
	\begin{center}
		\begin{tikzpicture}
			
			\begin{scope}[xshift=-1cm]
				%contorno
				\draw (2.5,2.8) ellipse (1.2 and 0.2); 
				\draw (0,-2.5) ellipse (1.2 and 0.2); 
				\draw (5,-2.5) ellipse (1.2 and 0.2); 
				\draw (1.2,-2.5) to [out=75, in=180] (2.5,-1) to [out=0, in=105] (3.8, -2.5); 
				\draw (-1.2,-2.5) to [out=60, in=270] (1.3, 2.8); 
				\draw (6.2,-2.5) to [out=120, in=270] (3.7, 2.8); 
				%	\draw (1.2, 0) to [out=340, in=200] (3.8, 0); 
				%ombreggiatura
				\shade[top color=white, bottom color=gray] (-1.2,-2.5) arc (180:360:1.2 and 0.2)to [out=75, in=180] (2.5,-1) to [out=0, in=105] (3.8, -2.5) arc (180:360:1.2 and 0.2) to [out=120, in=270] (3.7, 2.8) arc (0:-180:1.2 and 0.2) to [out=270, in=60] (-1.2,-2.5); 
				%archi tratteggiati
				\draw[densely dashed, very thin] (-1.2,-2.5) arc (180:0:1.2 and 0.2); 
				\draw[densely dashed, very thin] (3.8, -2.5) arc (180:0:1.2 and 0.2); 
				%asse theta
				\draw[->, very thin, gray] (9,-4) -- (9,4.5) node [black, below right] {$\theta$}; 
				\draw[((-)), thick] (9,-2.15)--(9,2.4); 
				\draw[(-), thick] (9,-1.5)--(9,1.8); 
				%graffe
				\draw[decorate, decoration={brace,mirror}] (9.4,-2.05)--(9.4,2.3) node[midway, right] {$J$}; 
				\draw[decorate, decoration={brace}] (8.8,-1.4)--(8.8,1.7) node[midway, left] {$I$}; 
				%punti
				\draw[fill=black] (3.3, 1.35) circle (0.05) node [left]{$z_1$}; 
				\draw[very thin, dashed] (3.3, 1.35) -- (9,1.35); 
				\draw[fill=black] (4.4, -1.8) circle (0.05) node [left]{$z_2$}; 
				\draw[very thin, dashed] (4.4, -1.8) -- (9,-1.7); 
				%Y
				\node at (-2,1.55)	{$Y$}; 	
			\end{scope}
			
			\begin{scope}[scale=1.5, xshift=-2cm, yshift=-6cm]
				
				\filldraw[fill=white!80!black] (0,0)--(35:1) circle (1);
				\draw (0.3,0) arc (0:35:0.3) node[right=0.3cm] {$\theta(z_1)$};  
				\node at (2.5,2.5) [below left] {$\C_{z_1}$}; 
				\draw[->, very thin, gray] (-1.5,0) -- (2.5,0) node[right, black] (1) {$\Re(w)$}; 
				\draw[->, very thin, gray] (0,-1.5) -- (0,2.5) node[above, black] (2) {$\Im(w)$};
				\begin{pgfonlayer}{background}
					\filldraw[very thin, fill=white, draw=gray] (-1.5,-1.5) rectangle (1.east |- 2.north); 
				\end{pgfonlayer}
				\begin{scope}[xshift=5cm]
					\filldraw[fill=white!80!black] (0,0)--(75:1) circle (0.6);
					\draw (0.3,0) arc (0:75:0.3) (0.3,0.2) node[right] {$\theta(z_2)$};  
					\node at (2.5,2.5) [below left] {$\C_{z_2}$}; 
					\draw[->, very thin, gray] (-1.5,0) -- (2.5,0) node[right, black] (1) {$\Re(w)$}; 
					\draw[->, very thin, gray] (0,-1.5) -- (0,2.5) node[above, black] (2) {$\Im(w)$};
					\begin{pgfonlayer}{background}
						\filldraw[very thin, fill=white, draw=gray] (-1.5,-1.5) rectangle (1.east |- 2.north); 
					\end{pgfonlayer}
				\end{scope}
			\end{scope}
		\end{tikzpicture}
	\end{center}
	\caption{A Worm, whose underlying Riemann surface $Y$ (depicted above) has genus zero and three boundary components. In this picture, the harmonic angle function $\theta$ is represented as a height function for visual clarity. In the two boxes below one finds a generic $w$-slice of the worm over a point $z_1\in \theta^{-1}(I)$ (on the left) and $z_2\in \theta^{-1}(J\setminus I)$ (on the right). Notice that, because of the indicated choice of $I$ and $J$, the surfaces $X_{\mathrm{in}}$ and $X_{\mathrm{out}}$ have the same topology (albeit in general different conformal structures).}
	\label{fig:worm}
\end{figure}
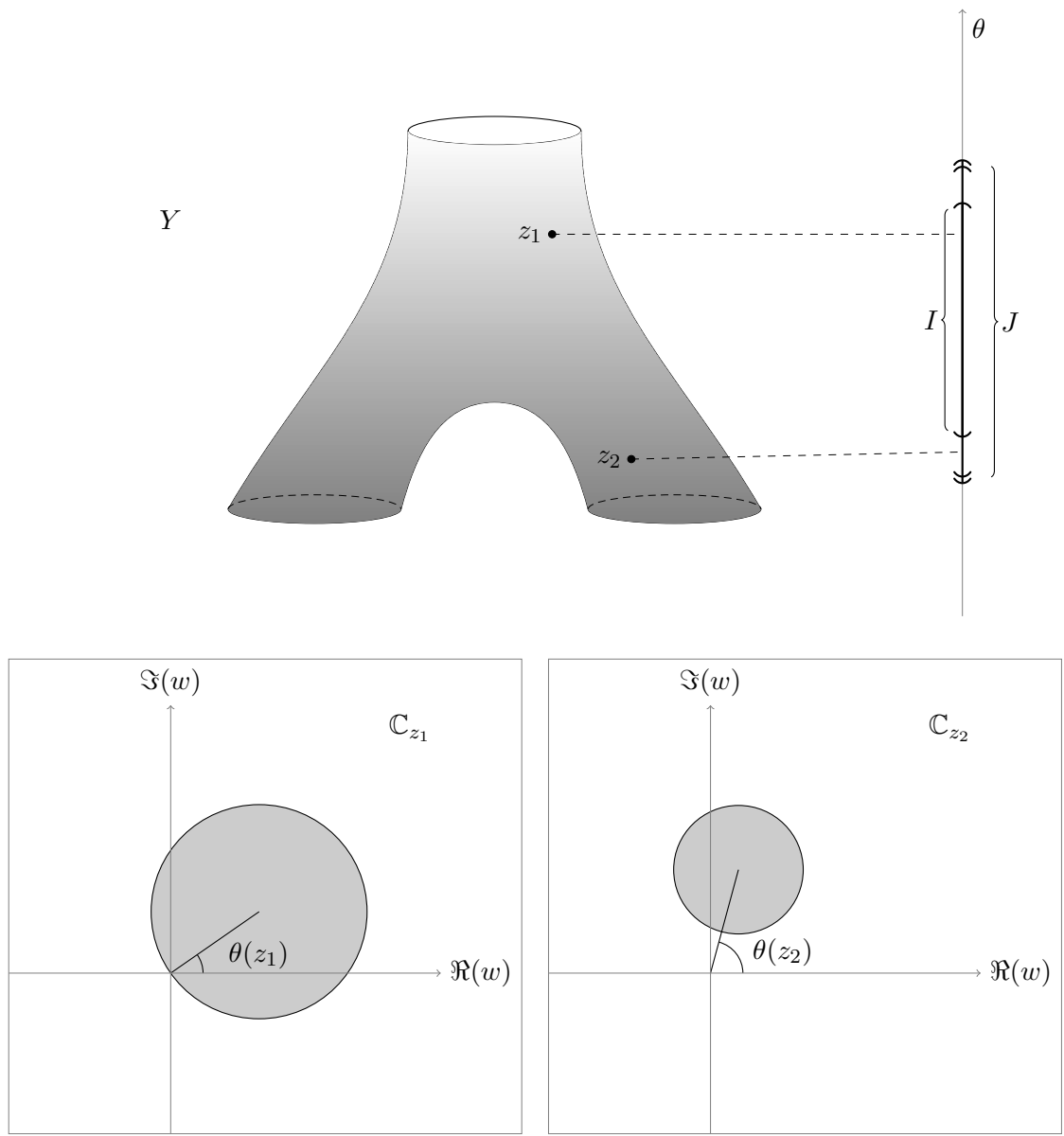

\begin{rmk}
By Docquier--Grauert \cite{DG} every Worm is Stein.
\end{rmk}
For a more refined analysis, we split the boundary of $W$ into four regions: \begin{itemize}
	\item the {\sl spine} of the Worm \[
	S := \{(z,w)\in \partial W\colon \theta(z)\in I, \ w=0\},
	\]	
	\item the {\sl body} of the Worm \[
B := \{(z,w)\in \partial W\colon \theta(z)\in I, \ \partial_z\theta(z)\neq 0, \ w\neq 0\},
	\]
	\item the  {\sl exceptional} set \[
	E := \{(z,w)\in \partial W\colon\ \partial_z\theta(z)=0, \ w\neq 0\},
	\]
	\item the {\sl caps} \[
	C:=\{(z,w)\in \partial W\colon \theta(z)\in J\setminus I, \ \partial_z\theta(z)\neq 0\}.
	\]
	
\end{itemize}
%The condition $(z,w)\in \partial W$ is implicit in all the definitions above. 

\begin{rmk}\label{X_is_smoothly_bounded}
Identify the slice $\{w=0\}\subset Y\times \C$ with the Riemann surface $Y$. Inside $Y$ the spine $S$ is the closure of the domain  \[
X_{\mathrm{in}}:=\theta^{-1}( I^\circ)\subset \subset Y.\] Since the angle function $\theta$ has no critical point $z\in \theta^{-1}(\partial I)$, the domain $X_{\mathrm{in}}$ is smoothly bounded.  $X_{\mathrm{in}}$ is a Riemann surface contained in the boundary of the Worm $W$, hence every point of the spine $S$ is of D'Angelo infinite type.

 In what follows an important role is also played by the smoothly bounded domain 
\[X_{\mathrm{out}}:=\theta^{-1}( J^\circ)\subset \subset Y.\] 
\end{rmk}

\begin{rmk}
The classical Worm domains introduced by Diederich--Forn\ae ss \cite{DF} correspond to the case where $Y=\C^*$, $\theta(z)=\log |z|^2$. In this case $X_{\mathrm{in}}$ is a holomorphic annulus contained in the boundary of $W$, whose conformal class depends on the choice of the interval $I$. 

A "genus zero" generalization of the Diederich--Forn\ae ss Worms is obtained choosing $Y=\C\setminus \{a_1,\ldots, a_k\}$ and $\theta(z)=\sum_{j=1}^k\lambda_j\log |z-a_j|^2$ (where $\lambda_j>0$). If $I=[-a,b]$ with $a$ and $b$ large enough, the spine $S$ has $k+1$ boundary components. 
\end{rmk}

\begin{prop}\label{haspeak}
The caps $C$ and the body $B$ consist of strongly pseudoconvex points, the exceptional set $E$ consists of finite-type points, and the spine $S$ consists of infinite type points. 
	\end{prop}

\begin{proof} 
	We already remarked that $S$ consists of infinite type points. 
	
	In the proof of Proposition \ref{prop:worm}, we saw that the boundary of a worm has a local defining function admitting the representation \[
\tilde{r}=|e^{-\frac{F}{2}}w|^2+e^{-v}\eta\circ \theta + \mathrm{ph}, 
\]
where $\mathrm{ph}$ denotes a pluriharmonic function, and that \[
\Delta_z(e^{-v}\eta\circ \theta )\geq e^{-v} |\partial_{\bar{z}}\theta|^2 \eta''\circ \theta. 
\]
Since the latter quantity is positive on the caps (thanks to the strict convexity assumption on $\eta$), we conclude that the Worm is strictly pseudoconvex at every point of $C$ where $\partial_{{w}}$ is \emph{not   tangent} to the boundary, that is $\partial_{ w}r\neq 0$ (or, equivalently, the vector $(0,1)$ is not in the complex tangent to $\partial W$). If instead $\partial_{ w}r= 0$, then we have $\theta\in \partial J$ (cf. the beginning of the proof of Proposition \ref{prop:worm}) and we may exploit the strong plurisubharmonicity of $|e^{-\frac{F}{2}}w|^2$: \[
\partial_{w}\partial_{\bar{w}}|e^{-\frac{F(z)}{2}}w|^2 = |e^{-\frac{F(z)}{2}}|^2>0. 
\]
Thus, every point of $C$ is strongly pseudoconvex.

We now study points $(z,w)$ in the body $B$, where $\eta\circ \theta\equiv0$. Calculating the Levi form $\mathcal L_{(z,w)}\tilde{r}$ we  obtain, for $(a,b)\in \C^2$,
$$
\mathcal L_{(z,w)}\tilde{r}(a,b)=
\begin{pmatrix}
 \overline a&\overline b 
 \end{pmatrix}
 |e^{-\frac{F(z)}{2}}|^2
 \begin{pmatrix}
 |w|^2\frac{|F'(z)|^2}{4} & -\overline w \frac{\overline{F'(z)}}{2}\\
 -w \frac{F'(z)}{2}& 1
 \end{pmatrix}
 \begin{pmatrix}
  a\\  b
 \end{pmatrix}.
$$
Hence $L_{(z,w)}\tilde{r}(a,b)$ vanishes if and only if $(a,b)\in \C^2$ is a multiple of $(2,wF'(z))$. This readily shows that the Worm is strongly pseudoconvex at every boundary point of the body where $(2,wF'(z))$ \emph{is not in the  complex tangent} to the boundary. But a simple computation shows that the vector $(2,wF'(z))$ is never complex tangent to the boundary since  \[
(2\partial_z + wF'(z)\partial_w)\tilde{r} = wF'(z)e^{-F(z)}\neq 0.
\]
This shows that every point of the body $B$ is strongly pseudoconvex. 

We are left with the proof that every point of the exceptional set $E$ is of finite type. By the Cauchy--Riemann equations, the critical points of $\theta$ are the same as the critical points of the (locally defined) holomorphic function $F$, and hence they are isolated. Thus, $E$ is a finite union of circles and circles with one point deleted (the point with $w=0$, in case the circle crosses the spine). Moreover, since $\theta$ has no critical points on $\partial I$, the boundary of the Worm is real-analytic in an open neighborhood of $E$. Thus, to verify that every point of $E$ is of finite type, we need to check that no positive dimensional complex analytic variety lies in such a neighborhood (see, e.g., \cite{BER}). This is easy, because any point of such a variety would be of infinite type and, since we already checked that $B$ and $C$ consist of strongly pseudoconvex points, this would force the variety to be contained in $E$, which is impossible by dimension considerations (or by the open mapping theorem). 
\end{proof}

\begin{rmk}
The Worms are examples of domains with nontrivial, yet nicely behaved, Levi core. See \cite{DM} and \cite{DM2}, where this notion has been introduced by the second-named author and S. Mongodi. As a consequence of Proposition \ref{haspeak}, the Levi core of a Worm is the $T^{1,0}$ bundle of its spine. A straightforward computation using \cite[Proposition 4.1, part vi)]{DM} shows that the de Rham cohomology class on the spine $S$ (or, equivalently, $X_{\mathrm{in}}$) induced by the D'Angelo class of the Worm is represented by $i(\overline\partial-\partial)\theta$, which is exact if and only if the angle function $\theta$ is globally on $X_{\mathrm{in}}$ the real part of a holomorphic function (that is, the pre-Worm $Z(X_{\mathrm{in}}, \theta|_{X_{\mathrm{in}}})$ is trivial as a fiber bundle). This is in turn equivalent to the condition that the Diederich--Forn\ae ss index of the Worm is $1$. We refer to \cite[Section 4]{DM} for a review of the basic theory of D'Angelo classes and to \cite{AY, DM} for the implications on the Diederich--Forn\ae ss index. 
	\end{rmk}

We end this section proving that Worms are complete Kobayashi hyperbolic. For this, we observe that a Worm $W$ is naturally associated with two pre-Worms.
\begin{dfn} 
Set  $$W_{\mathrm{in}}:=Z(X_{\mathrm{in}},\theta|_{X_{\mathrm{in}}}),\quad W_{\mathrm{out}}:=Z(X_{\mathrm{out}},\theta|_{X_{\mathrm{out}}}),$$
where we are using the notation of Remark \ref{X_is_smoothly_bounded}. Notice that $W_{\mathrm{in}}\subset W_{\mathrm{out}}$ and  $W\subset W_{\mathrm{out}}$.
\end{dfn}

In the remaining of the paper, if $M$ is a complex manifold, we denote by $k_M$ its Kobayashi pseudodistance and by $K_M$ its Kobayashi--Royden pseudometric. The following  lemma is proved in \cite[Lemma 2.1.3]{gaussier-taut}.

\begin{lem}\label{gausstaut}
Let $D\subset\C^d$ be a domain and $k_D$ its Kobayashi distance. If $z_n\rightarrow\xi\in\partial D$ and $\xi$ admits a local holomorphic peak function then for every neighborhood $U$ of $\xi$ we get
$$\lim_{n\to+\infty}k_D(z_n,D\cap U^c)=+\infty.$$
\end{lem}

\begin{prop}\label{kobcompl}
Worms are complete Kobayashi hyperbolic. 
\end{prop}

\begin{proof}
Assume by contradiction that there exists a nonconvergent Cauchy sequence $\{x_n\}_n$ in $W$. Passing to a subsequence, we can assume that $x_n\rightarrow \xi\in\partial W$. We write $x_n=(z_n,w_n)$ and $\xi=(z_0,w_0)$. 

If $w_0\neq 0$, then  $\xi$ is a pseudoconvex finite type point by Proposition \ref{haspeak}.  By \cite{bedfor} (see also \cite[Section 4]{noell}) $\xi$ admits a local holomorphic peak function and hence it can not be a Cauchy sequence by Lemma \ref{gausstaut}.  

Assume next that $w_0=0$, so that in particular $\xi\in \partial W_{\mathrm{out}}$. Since $W\subset W_{\mathrm{out}}$, it follows that $\{x_n\}_n$ is also a Cauchy sequence w.r.t.~$k_{W_{\mathrm{out}}}$, which converges to $\xi\in\partial W_{\mathrm{out}}$. This contradicts the completeness of the pre-Worm $W_{\mathrm{out}}$ (Proposition \ref{prop:bundles_complete} below).
\end{proof}

\section{Holomorphic fiber bundles are not Gromov hyperbolic}

We recall  a  classical result from the theory of Kobayashi hyperbolic complex manifolds. If $M$ is a complex manifold, we denote by  $B_M(p,r)$ the $k_M$-ball of center $p$ and radius $r$.
\begin{prop}[\cite{kobayascione}, Proposition 3.1.19]\label{lem:bilipschitz}
Let $M$ be a Kobayashi hyperbolic complex manifold.  Let $p\in M$ and $R, \epsilon>0$. Then there exists a constant $C\geq 1$ depending only on  $\epsilon$ such that
$$k_{B_M(p,3R+\epsilon)}(x,y)\leq C k_M(x,y), \quad \forall x,y\in B_M(p,R),$$
and thus the metrics $k_M$ and $k_{B_M(p,3R+\epsilon)}$ are biLipschitz equivalent on $B_M(p,R)$.
\end{prop}

The fact that $C$ depends only on $\epsilon$ is not stated explicitly in \cite[Proposition 3.1.19]{kobayascione}, but it is clear from the (first paragraph of the) proof. We will actually use this result in the following simplified form. 
\begin{cor}\label{corbilip}
Let $M$ be a Kobayashi hyperbolic complex manifold. Then there exists an absolute constant $C\geq 1$ such that
$$k_{B_M(p,4R)}(x,y)\leq C k_M(x,y), \quad \forall x,y\in B_M(p,R),$$
for all $R\geq 1$ and all $p\in M.$
\end{cor}

We introduce the following definition.
\begin{dfn}
Let $\pi:E\rightarrow X$ be a holomorphic fiber bundle and $z\in X$. Then define\[
r(z):=\sup\{r>0\colon \ \text{ the bundle trivializes over } B_X(z,r)\}
\]
\end{dfn}
Notice that $r(z)>0$ for every $z\in X$ if $X$ is Kobayashi hyperbolic.

We can now prove the main result of this section. Recall \cite[Theorem 3.1.9]{kobayascione}  that if $X$ and $Y$ are two complex manifolds then
\begin{equation}\label{product}k_{X\times Y}((z_1,w_1),(z_2,w_2))=\max\left\{k_X(z_1,z_2),k_Y(w_1,w_2)\right\},\quad z_1,z_2\in X,w_1,w_2\in Y.\end{equation}

\begin{thm}\label{lem:bundles_not_gromov} 
	Let $X,F$ be non-compact complete Kobayashi hyperbolic complex manifolds.
	Let $\pi:E\rightarrow X$ be a holomorphic fiber bundle with fiber $F$ and such that $\sup_{z\in X}r(z)=+\infty$. Then $(E,k_E)$ is not Gromov hyperbolic. 
\end{thm}
\begin{proof} 
We will construct a sequence $\{T_n\}_n$ of  quasigeodesic triangles in $E$ violating the definition of Gromov hyperbolicity. Let $\{z_n\}_n$ in $X$ be such that $r_n:=r(z_n)\to+\infty$. We define 
$$\Omega_n:=\pi^{-1}(B_X(z_n,r_n/2)),$$ and let 
$$\Psi_n\colon B_X(z_n,r_n/2)\times F\to \Omega_n$$ be a holomorphic trivialization. Let $q\in F$ be  any point of $F$. Let $C\geq 1$ be the universal constant given by Corollary \ref{corbilip}. Set $t_n:=\frac{r_n}{16C}$. 

We construct the triangles in the following way. Since $X$ and $F$ are non-compact, for all $n>0$ we can find a geodesic of $X$ denoted $\gamma_n:[0,t_n]\rightarrow X$ with $\gamma_n(0)=z_n$, and a geodesic of $F$ denoted $\sigma_n:[0,t_n]\rightarrow F$ with $\sigma_n(0)=q$. Notice that $\gamma_n([0,t_n])\subset B_X(z_n,r_n/8)$, so by Corollary \ref{corbilip} the curve $\gamma_n$ is a $(C,0)$-quasigeodesic w.r.t.~the Kobayashi distance of $B_X(z_n,r_n/2)$ (we may assume that $r_n\geq 8$ for every $n$).

By \eqref{product} the curves  $a_n(t)=(z_n,\sigma_n(t))$ and $b_n(t)=(\gamma_n(t),q)$ are respectively  a geodesic and $(C,0)$-quasigeodesic of $B_X(z_n,r_n/2)\times F$. Moreover a simple computation shows that the curve  $c_n:[0,2t_n]\rightarrow B_X(z_n,r_n/2)\times F$ defined by 
$$c_n(t)=\begin{cases}(\gamma_n(t),\sigma_n(t_n)) &\mbox{if}\ t\in[0,t_n]\\(\gamma_n(t_n),\sigma_n(2t_n-t))&\mbox{if}\ t\in[t_n,2t_n].\end{cases}$$
 is a $(2C,0)$-quasigeodesic of $B_X(z_n,r_n/2)\times F$. Indeed, $c_n$ is a geodesic w.r.t. the distance $k_X+k_F$ that is $2C$-BiLipschitz to $k_{B_X(z_n,r_n/2)\times F}$ in $B_X(z_n,r_n/8)\times F$. Hence the triangle $T_n$ with sides $a_n, b_n$ and $c_n$  is a $(2C,0)$-quasigeodesic triangle in $B_X(z_n,r_n/2)\times F$. Notice $T_n$ is not $t_n$-slim  because $$k_{B_X(z_n,r_n/2)\times F}(c_n(t_n),a_n([0,t_n])\cup b_n([0,t_n]))=t_n.$$ 
 Now since $\Psi_n$ is a biholomorphism between $B_X(z_n,r_n/2)\times F$ and $\Omega_n$,  the triangle $\hat{T}_n$ in $\Omega_n$ that is image of $T_n$ via $\Psi_n$ is again a $(2C,0)$ quasi-geodesic triangle w.r.t. $k_{\Omega_n}$, and it is not $t_n$-slim.
 \begin{figure}[ht]
	\begin{tikzpicture}[scale=2,  mydot/.style={circle, fill=black, draw, outer sep=0pt,inner sep=1.3pt}]
	\draw[line width=.5mm] (-2,-2)--node[below] {$B_X(z_n,r_n/2)$}(2,-2);
	\draw[line width=.5mm] (-2,-2)--node[left] {$F$}(-2,0.5);
	\draw[line width=.5mm] (2,-2)--(2,0.5);
	
	\draw[line width=.5mm] (0,-1.25)--node[above] {$b_n$}(1.25,-1.25);
	\draw[line width=.5mm] (0,-1.25)--node[right] {$a_n$}(0,0);
	\draw[line width=.5mm] (1.25,-1.25)--(1.25,0);
	\draw[line width=.5mm] (0,0)--(1.25,0) node[below left] {$c_n$};
	
	\draw[very thin, dashed] (0,-1.25)--(0,-2) node [above right] {$z_n$}; 
	\draw[very thin, dashed] (1.25,-1.25)--(1.25,-2) node [above right] {$\gamma_n(t_n)$}; 
	
	\draw[very thin, dashed] (1.25,-1.25)--(2,-1.25) node [above right] {$q$}; 
	\draw[very thin, dashed] (1.25,0)--(2,0) node [above right] {$\sigma_n(t_n)$}; 
	
	\node at (-2,0.7) {$\vdots$};
	\node at (2,0.7) {$\vdots$};

	\node[mydot]  at (0,-2) {};
	\node[mydot]  at (1.25,-2) {};
	\node[mydot]  at (2,-1.25) {};
	\node[mydot]  at (2,0) {};
	
	\node[mydot]  at (0,0) {};
	\node[mydot]  at (0,-1.25) {};
	\node[mydot]  at (1.25,-1.25) {};

	\end{tikzpicture}
	\caption{The triangle $T_n$ in $B_X(z_n,r_n/2)\times F$.}
\end{figure}
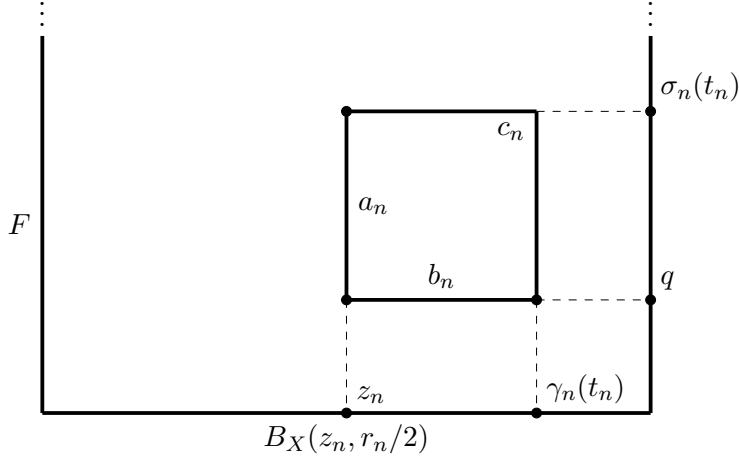

Now the map $\pi\colon E\to X$ is non-expanding,
so  $$ B_E(\Psi(z_n,q),r_n/2)\subset \Omega_n.$$
The triangle $\hat{T}_n$ is contained in $B_{\Omega_n}(\Psi(z_n,q),r_n/8)$, and hence it is contained in $B_E(\Psi(z_n,q),r_n/8)$. By another application of Corollary \ref{corbilip}, the distances $k_E$ and $k_{\Omega_n}$ are $C$-BiLipschitz in $B_E(\Psi(z_n,q),r_n/8)$, so $\hat{T}_n$ is a $(2C^2,0)$-quasigeodesic triangle not $(C^{-1}t_n)$-slim w.r.t.~the distance $k_E$. It follows that  $E$ is not Gromov hyperbolic.
\end{proof}

We conclude this section highlighting an interesting class of holomorphic fiber bundles satisfying the condition $\sup_X r=+\infty$.

\begin{prop}
Let $Y$ be a complex manifold  and let $\pi:E\rightarrow Y$ be a holomorphic fiber bundle.
Let  $X\subset Y$ be a  domain. Assume that there exists a point $\xi\in \partial X$ which admits a local holomorphic peak function. 
Then the restricted holomorphic bundle $E|_X$  has the property $\sup_X r=+\infty$.
\end{prop}

\begin{proof}
Let $U$ be an open  neighborhood of $\xi$ in $Y$  such that $\pi:E\rightarrow Y$ trivializes over $U$.  Let $\{z_n\}_n$ be a sequence in $X$ converging to $\xi$. 
By Lemma \ref{gausstaut} we have that
$$\lim_{n\to+\infty}k_X(z_n,X\cap U^c)=+\infty.$$
Hence for each $R>0$ we have $B_X(z_n,R)\subset X\cap U$ for $n$ large enough, which implies that $r(z_n)\rightarrow +\infty.$
\end{proof}

\begin{cor}\label{winnotg}
The pre-Worms $W_{\mathrm{in}}$ and $W_{\mathrm{out}}$ are not Gromov hyperbolic w.r.t.~its Kobayashi distance.
\end{cor}

\begin{proof}
The  domains $$X_{\mathrm{in}}\subset \subset X_{\mathrm{out}}\subset \subset Y$$ are smoothly bounded (see Remark \ref{X_is_smoothly_bounded}), and thus every point in their boundary admits a local holomorphic peak function. Hence by the previous proposition 
the pre-Worms $W_{\mathrm{in}}$ and $W_{\mathrm{out}}$ satisfy $\sup_X r=+\infty$ and Theorem \ref{lem:bundles_not_gromov} yields the result.
\end{proof}

\section{Barrett's scaling and proof of main theorem}

In what follows  we denote by $TM$ the holomorphic tangent bundle of a complex manifold $M$ and by $\pi:TM\rightarrow M$ the canonical projection. We denote by $\D\subset \C$ the unit disc. Recall the following classical definition.
\begin{dfn}Let $M$ be a complex manifold and let $X\subset M$ be a domain. Then $X$ has {\sl simple boundary} in $M$ if for all $\phi:\D\rightarrow M$  holomorphic mappings with $\phi(\D)\subset \overline{X}$ and $\phi(\D)\cap \partial X\neq \varnothing$ one has $\phi(\D)\subseteq \partial X$.
\end{dfn}

The proof of Theorem \ref{main} is based on the following result, showing the stability of the Kobayashi distance and of the Kobayashi--Royden metric under a  particular type of convergence of domains $D_n\to D_\infty$.
\begin{prop}
\label{lem:continuity_metric}
Let $M$ be a taut complex manifold and let $\{ D_n\}_n$ be a sequence of domains of $M$. 
Let  $D_\infty\subset M$ be a complete Kobayashi hyperbolic domain with simple boundary. Assume that 
\begin{itemize}
  \item[i)] if $\{ x_n\}_n$ is a sequence converging to $x_\infty\in M$ and $x_n\in D_n$ for all $n\in \N$, then $x_\infty\in \overline{D}_\infty$;
  \item[ii)] for every compact $H\subset D_\infty$, there exists $N$ such that $H\subset D_n$ for $n\geq N$.
\end{itemize}
Then as $ n\rightarrow +\infty$ we have 
$K_{D_n}\rightarrow K_{D_\infty}$ uniformly on compact subsets of  $TD_\infty$, and
$k_{D_n}\rightarrow k_{D_\infty}$ uniformly on compact subsets of $D_\infty\times D_\infty$.
\end{prop}
See, e.g., \cite[Chapter 5]{kobayascione} for the notion of tautness.
The idea of the proof of Proposition \ref{lem:continuity_metric} is similar to \cite[Theorem 4.3]{BGZ}.
 The proof  is based on two lemmas, valid under the assumptions of the proposition.

\begin{lem}\label{lem:lower_continuity}
For every $H\subset D_\infty$ compact and $\epsilon>0$, there exists $N$ such that for all $n\geq N$ and for all $v\in \pi^{-1}(H)$ we have
$$
K_{D_n}(v)\leq (1+\epsilon)K_{D_\infty}(v).
$$
\end{lem}
\begin{proof}
Set $r:=(1+\epsilon)^{-1}\in(0,1)$. Define $\widehat{H}\subset D_\infty$ as
$$\widehat{H}:=\{\phi(\zeta)|\, \phi:\D\rightarrow D_\infty \text{ holomorphic},\ \phi(0)\in H,\ |\zeta|\leq r \}.$$
The set $\widehat{H}$ is compact. Indeed, let $\{z_n\}_n$ be a sequence in $\widehat{H}$, i.e., there exist $\phi_n:\D\rightarrow D_\infty$ such that $\phi_n(0)\in H$, and $|\zeta_n|\leq r$ such that $z_n=\phi_n(\zeta_n). $
Since $\phi_n(0)\in H$ for all $n\in \mathbb{N}$ and $D_\infty$ is taut (by \cite[Theorem 5.1.3]{kobayascione}), we can assume that $\phi_n$ converges uniformly on compact sets to a holomorphic map $\hat{\phi}:\D\rightarrow D_\infty$ and that  $\zeta_n$ converges to $\hat{\zeta}$ with $|\hat{\zeta}|\leq r$. But then $z_n\to \hat{\phi}(\hat{\zeta})\in\widehat{H}$. This proves that $\widehat{H}$ is compact.

Now, for each $v\in \pi^{-1}(H)$ let $\phi:\D\rightarrow D_\infty$ be such that $\phi(0)=\pi(v)$ and $$K_{D_\infty}(v)\phi'(0)=v.$$ Using property (ii) there exists $N$ such that for all $n\geq N$ we have $\widehat{H}\subset D_n$, which implies that if $\phi_r:\D\rightarrow D_\infty$ is defined by $\phi_r(z):=\phi(r z)$ then $\phi_r(\D)\subset \widehat{H}\subset D_n$.
Finally, using the definition of the Kobayashi--Royden metric we have
$$K_{D_n}(v)\leq r^{-1}K_{D_\infty}(v)=(1+\epsilon)K_{D_\infty}(v). $$
\end{proof}

\begin{lem}\label{lem:upper_continuity}
	For every $H\subset D_\infty$ compact and $\epsilon>0$, there exists $N$ such that for all $n\geq N$ and for all $v\in \pi^{-1}(H)$ we have
	\begin{equation}\label{stima2}
		K_{D_\infty}(v)\leq (1+\epsilon)K_{D_n}(v).
	\end{equation}
\end{lem}
\begin{proof}
	Fix an Hermitian metric on $TD_\infty$. The result immediately follows if we prove \eqref{stima2} for all $v\in\pi^{-1}(H)$ such that $\|v\|=1$.
	Assume by contradiction that there exist $H\subset D_\infty$ compact, $\epsilon>0$ and $n_k\rightarrow+\infty$, $v_k\in \pi^{-1}(H)$ such that $\|v_k\|=1$ and
	$$K_{D_\infty}(v_k)> (1+\epsilon)K_{D_{n_k}}(v_k).$$
	We can assume that $v_k\rightarrow v_\infty\in \pi^{-1}(H)$.
	Let $\phi_k:\D\rightarrow D_{n_k}$ be a holomorphic map such that $\phi_k(0)=\pi(v_k)$ and $\alpha_k\phi_k'(0)=v_k$, where $\alpha_k\leq (1+\epsilon)^{1/2}K_{D_{n_k}}(v_k)$. In particular, $\alpha_k\leq (1+\epsilon)^{-1/2}K_{D_{\infty}}(v_k)$ and hence $\alpha_k$ is uniformly bounded in $k$. We may therefore assume that $\alpha_k$ converges to a limit $\alpha$ as $k\rightarrow +\infty$.
	
	Since $M$ is taut and $\phi_k(0)\in H$, we can assume that the sequence $\{\phi_k\}_k$ converges uniformly on compact sets to a holomorphic map $\phi:\D\rightarrow M$, which satisfies the identity $\alpha \phi'(0)=v_\infty$. Using property (i)  we have $\phi(\D)\subset\overline{D}_\infty$. Since $D_\infty$ has simple boundary in $M$ it follows from  $\phi(0)=\pi(v_\infty)\in D_\infty$ that $\phi(\D)\subset D_\infty$. Finally using the definition of the Kobayashi--Royden metric we have 
	$$K_{D_\infty}(v_\infty)\leq \alpha\leq \lim_k(1+\epsilon)^{-1/2}K_{D_\infty}(v_k)=(1+\epsilon)^{-1/2}K_{D_\infty}(v_\infty),$$
	which is a contradiction. 
\end{proof}

\begin{proof}[Proof of Proposition \ref{lem:continuity_metric}] The uniform convergence on compact subsets of the Kobayashi--Royden metric follows from Lemmas \ref{lem:lower_continuity} and \ref{lem:upper_continuity}. We now prove the local uniform convergence of the Kobayashi distance. In what follows, we denote by $\ell_M(\gamma)$ the Kobayashi--Royden length of a curve $\gamma$ on the manifold $M$. 
	
Let $H\subset D_\infty$ be  a compact set, and set $R:=\mathrm{diam}(H)$. Given $p,q\in H$ and $\epsilon\in (0,1)$, let $\gamma:[0,1]\rightarrow D_\infty$ be a piecewise $C^1$ curve joining $p$ with $q$ and satisfying $\ell_{D_\infty}(\gamma)\leq k_{D_\infty}(p,q)+\epsilon$. Fix $o\in H$. Then, for all $t\in[0,1]$,
$$k_{D_\infty}(o,\gamma(t))\leq k_{D_\infty}(o,p)+k_{D_\infty}(p,\gamma(t))\leq R+\ell_{D_\infty}(\gamma)\leq R+k_{D_\infty}(p,q)+\epsilon\leq 2R+1, $$
i.e., the support of $\gamma$ is contained in $ \overline{B_{D_\infty}(o,2R+1)}$ which is a compact subset of $D_\infty$ by the completeness of $D_\infty$. By Lemma \ref{lem:lower_continuity} there exists $N$ such that for all $n\geq N$ and for all $v\in \pi^{-1}(\overline{B_{D_\infty}(o,2R+1)})$ we have
$
K_{D_n}(v)\leq (1+\epsilon)K_{D_\infty}(v),
$
which implies $\ell_{D_n}(\gamma)\leq(1+\epsilon)\ell_{D_\infty}(\gamma)$. Hence
\[
k_{D_n}(p,q)\leq\ell_{D_n}(\gamma)\leq(1+\epsilon)\ell_{D_\infty}(\gamma)\leq (1+\epsilon)(k_{D_\infty}(p,q)+\epsilon)\leq k_{D_\infty}(p,q)+O((1+R)\epsilon).
\]
In particular, 
\begin{equation}\label{weneedthis}
	k_{D_n}(p,q)=O(1+R)
\end{equation}
for $n\geq N$. 

For the converse, notice that by (ii) $H$ is eventually contained in the domains $D_n$. Given $p,q\in H$ and $\epsilon\in (0,1)$, let $\gamma_n:[0,1]\rightarrow D_n$ be a piecewise $C^1$ curve joining $p$ with $q$ and satisfying $\ell_{D_n}(\gamma_n)\leq k_{D_n}(p,q)+\epsilon$. Fix $o\in H$ and define 
$$t_n:=\sup\{t\in[0,1]:\gamma_n([0,t])\subset B_{D_\infty}(o,2R)\}.$$ We have that $k_{D_\infty}(p,\gamma_n(t_n))\geq k_{D_\infty}(p,q)$.
Indeed, this clearly holds if $t_n=1$. If $t_n<1$, then $k_{D_\infty}(o,\gamma_n(t_n))=2R$ and thus $$k_{D_\infty}(p,\gamma_n(t_n))\geq k_{D_\infty}(o,\gamma_n(t_n))-k_{D_\infty}(p,o)\geq 2R-R=R\geq k_{D_\infty}(p,q).$$
 
Since $\overline{B_{D_\infty}(o,2R)}$ is compact, by Lemma \ref{lem:upper_continuity} there exists $N$ such that for all $n\geq N$ and for all $v\in \pi^{-1}(\overline{B_{D_\infty}(o,2R)})$ we have that $
K_{D_n}(v)\geq (1+\epsilon)^{-1}K_{D_\infty}(v).$ Hence 
\begin{align*}k_{D_n}(p,q)+\epsilon&\geq \ell_{D_n}(\gamma_n)\geq\ell_{D_n}(\gamma_n|_{[0,t_n]})\geq (1+\epsilon)^{-1}\ell_{D_\infty}(\gamma_n|_{[0,t_n]})\\
&\geq (1+\epsilon)^{-1}k_{D_\infty}(p,\gamma_n(t_n))\geq (1+\epsilon)^{-1}k_{D_\infty}(p,q),
\end{align*}
that is
$$k_{D_\infty}(p,q)\leq (1+\epsilon)(k_{D_n}(p,q)+\epsilon)\leq k_{D_n}(p,q)+O((1+R)\epsilon),
$$
where we used  \eqref{weneedthis}.
\end{proof}

 Let $W$ be a  Worm.  We call {\sl Barrett's scaling} the one-parameter group of automorphisms of $Y\times \C$ given by  
\[\Barr_\lambda:(z,w)\mapsto (z, \lambda w) \qquad (\lambda>0),\] 
which played a key role in \cite[Section 4]{Barrett}. 

For all $n\geq 1$ we set $D_n:= \Barr_n(W)$, $D_\infty:= W_{\mathrm{in}},$ and $M:= W_{\mathrm{out}}.$
\begin{rmk}
Properties i) and ii) of Proposition \ref{lem:continuity_metric} are satisfied in this case.
\end{rmk}
\begin{lem}
The domain $W_{\mathrm{in}}$ has simple boundary in  $W_{\mathrm{out}}$.
\end{lem}
\begin{proof}
Let $\varphi\colon \D\to W_{\mathrm{out}}$ be a holomorphic map such that 
$\varphi(\D)\subset \overline W_{\mathrm{in}}$.  Assume that there exists $\zeta_0\in\D$ such that $$(z_0,w_0):=\phi(\zeta_0)\in \partial W_{\mathrm{in}}.$$
Clearly $z_0\in\partial X_{\mathrm{in}}$. If $\pi_1:X_{\mathrm{out}}\times\C\rightarrow X_{\mathrm{out}}$ denotes  the projection to the first variable, then $\pi_1\circ\phi:\D\rightarrow X_{\mathrm{out}}$ is a holomorphic function 
with image contained in $\overline X_{\mathrm{in}}$ and such that $(\pi_1\circ\phi)(\zeta_0)\in\partial X_{\mathrm{in}}$, hence by the open mapping theorem $\pi_1\circ\phi$ is constant. Thus $\phi(\D)\subset \partial W_{\mathrm{in}}$.\end{proof}

We are finally able to prove our main theorem.

\begin{proof}[Proof of Theorem \ref{main}]
By contradiction, assume that there exists $\delta\geq 0$ such that for each $o,x,y,z\in W$ we have
$$\min\{(x|y)^{k_W}_o,(y|z)^{k_W}_o\}-(x|z)^{k_W}_o\leq\delta.$$
Now since for all $n\geq 1$  the Barrett's scaling $B_n$ is an isometry between $W$ and $B_n(W)$ we have, for each  $o,x,y,z\in B_n(W)$, $$\min\{(x|y)^{k_{B_n(W)}}_o,(y|z)^{k_{B_n(W)}}_o\}-(x|z)^{k_{B_n(W)}}_o\leq\delta.$$
By Proposition \ref{lem:continuity_metric} we have,  for all $o,x,y,z\in W_{\mathrm{in}}$,
$$\min\{(x|y)^{k_{W_{\mathrm{in}}}}_o,(y|z)^{k_{W_{\mathrm{in}}}}_o\}-(x|z)^{k_{W_{\mathrm{in}}}}_o=\lim_{n\rightarrow+\infty}\min\{(x|y)^{k_{B_n(W)}}_o,(y|z)^{k_{B_n(W)}}_o\}-(x|z)^{k_{B_n(W)}}_o\leq\delta.$$
Thus  $W_{\mathrm{in}}$ is Gromov hyperbolic, which contradicts Corollary \ref{winnotg}.
\end{proof}

  \end{document}